\documentclass[10pt]{amsart}
\usepackage{amsmath, amsthm, amsfonts, amssymb, graphicx, bm,verbatim,mathrsfs,siunitx}
\usepackage{hyperref}
\usepackage{stmaryrd,enumerate,tikz}
\newtheorem{theorem}{Theorem}[section]
\newtheorem{lemma}[theorem]{Lemma}
\newtheorem*{remark}{Remark}
\newtheorem*{definition}{Definition}
\begin{document}

\title{Sums of two squares in short intervals}
\author{James Maynard}
\email{james.alexander.maynard@gmail.com}
%
%

\begin{abstract}We show that there are short intervals $[x,x+y]$ containing $\gg y^{1/10}$ numbers expressible as the sum of two squares, which is many more than the average when $y=o( (\log{x})^{5/9})$. We obtain similar results for sums of two squares in short arithmetic progressions.
\end{abstract}
\maketitle
\section{Introduction}
Let $X$ be large and $y\in (0,X]$. The prime number theorem shows that \textit{on average} for $x\in[X,2X]$ we have
\begin{equation}
\pi(x+y)-\pi(x)\approx \frac{y}{\log{x}}.\label{eq:ShortInterval}
\end{equation}
The approximation \eqref{eq:ShortInterval} is known (\cite{HeathBrownIntervals}, improving on \cite{Huxley}) to hold in the sense of asymptotic equivalence for \textit{all} $x\in[X,2X]$ provided that $X^{7/12}\le y\le X$, whilst the celebrated result of Maier \cite{Maier} shows that \eqref{eq:ShortInterval} does \textit{not} hold in the sense of asymptotic equivalence for all $x\in[X,2X]$ when $y$ is of size $(\log{X})^A$, for any fixed $A>0$. He showed that for $y=(\log{X})^A$, there exists $x\in[X,2X]$ such that
\begin{equation}
\pi(x+y)-\pi(x)>(1+\delta_A)\frac{y}{\log{x}}
\end{equation}
for a positive constant $\delta_A$ depending only on $A$ (and he also obtained similar results for intervals containing fewer than the average number of primes).

It is trivial that in the range $y\ll \log{X}$, \eqref{eq:ShortInterval} cannot hold in the sense of asymptotic equivalence, and the extent to which this approximation fails is closely related to the study of gaps between primes. Maier and Stewart \cite{MaierStewart} combined the Erd\H os-Rankin construction for large gaps between primes with the work of Maier to get show the existence of intervals containing significantly fewer primes than the average. It follows from the recent advances in the study of small gaps between primes (\cite{Zhang}, \cite{Maynard}, \cite{MaynardII} and unpublished work of Tao) that there are also short intervals containing significantly more primes than the average.
%

The situation is similar for the set $\mathcal{S}$ of integers representable by the sum of two squares. Hooley \cite{HooleyIV} has shown that for almost all $x\in[X,2X]$ one has
\begin{equation}
\frac{\mathfrak{S}y}{\sqrt{\log{X}}}\ll \sum_{n\in\mathcal{S}\cap[x,x+y]}1\ll \frac{\mathfrak{S}y}{\sqrt{\log{X}}},
\end{equation}
provided $y/\sqrt{\log{X}}\rightarrow\infty$ (here $\mathfrak{S}>0$ is an absolute constant so that the above sum is $(1+o(1))\mathfrak{S}y/\sqrt{\log{X}}$ on average over $x\in[X,2X]$).

By adapting Maier's method, Balog and Wooley \cite{BalogWooley} showed that for any fixed $A>0$ and  $y= (\log{X})^A$, there exists $\delta_A'>0$ and an $x\in[X,2X]$ such that
\begin{equation}
\sum_{n\in\mathcal{S}\cap[x,x+y]}1\ge (1+\delta_A')\frac{\mathfrak{S}y}{\sqrt{\log{X}}},
\end{equation}
and so Hooley's result cannot be strengthened to an asymptotic which holds for all $x\in[X,2X]$ when $y$ is of size $(\log{x})^A$. 

As a result of the `GPY sieve', Graham Goldston Pintz and Y\i ld\i r\i m \cite{GGPY} have shown that for any fixed $m$ there exists $x\in[X,2X]$ such that $[x,x+O_m(1)]$ contains $m$ numbers which are the product of two primes $\equiv 1\pmod{4}$, and so $m$ elements of $\mathcal{S}$. By modifying the GPY sieve to find short intervals containing many elements of $\mathcal{S}$, we show the existence of intervals $[x,x+y]$ containing a higher order of magnitude than the average number when $y$ is of size $o((\log{X})^{5/9})$. Thus in this range we see that the upper bound in Hooley's result cannot hold for all $x\in[X,2X]$. (It is a classical result that the lower bound cannot hold for all $x$ in such a range;  the strongest such result is due Richards \cite{Richards} who has shown there exists $x\in[X,2X]$ such that there are no integers $n\in\mathcal{S}$ in an interval $[x,x+O(\log{x})]$.)
\section{Intervals with many primes or sums of two squares}
Let the function $g:(0,\infty)\rightarrow\mathbb{R}$ be defined by
\begin{equation}
g(t)=\sup_{u\ge t}e^{\gamma}\omega (u),
\end{equation}
where $\gamma$ is the Euler constant and $\omega (u)$ is the Buchstab function
\footnote{The Buchstab function $\omega (u)$ is defined by the delay-differential equation \[\omega (u)=u^{-1}\quad (0<u\le 2),\qquad\qquad \frac{\partial}{\partial u}(u\omega(u))=\omega(u-1)\quad (u\ge 2).\]}
. Clearly $g$ is decreasing, and it is known that $g(t)>1$ for all $t>0$.

We first recall the current state of knowledge regarding short intervals and short arithmetic progressions containing unusually many primes
\begin{theorem}\label{theorem:Prime}
Fix $\epsilon>0$. There is a constant $c_0(\epsilon)>0$ such that we have the following:
\begin{enumerate}
\item Let $X,y$ satisfy $c_0(\epsilon)\le y\le X$. Then there exists at least $X^{1-1/\log\log{X}}$ values $x\in[X,2X]$ such that
\[\pi(x+y)-\pi(x)\ge\Bigl(g\Bigl(\frac{\log{y}}{\log\log{x}}\Bigr)-\epsilon+c\frac{(\log{x})(\log{y})}{y}\Bigr)\frac{y }{\log{x}}.\]
\item Fix $a\in\mathbb{N}$, and let $Q,x$ satisfy $c_0(\epsilon)\le Q\le x/2$. Then there exists at least $Q^{1-1/\log\log{Q}}$ values $q\in [Q,2Q]$ such that
\[\pi(x;q,a)\ge \Bigl(g\Bigl(\frac{\log{x/q}}{\log\log{x}}\Bigr)-\epsilon+c\frac{(\log{x})(\log{x/q})}{x/q}\Bigr)\frac{x}{\phi(q)\log{x}}.\]
\item Let $q,x$ satisfy $c_0(\epsilon)\le q\le x/2$, and let $q$ have no prime factors less than $\log{x}/\log\log{x}$. Then there exists at least $q^{1-1/\log\log{q}}$ integers $a\in[1,q]$ such that
\[\pi(x;q,a)\ge \Bigl(g\Bigl(\frac{\log{x/q}}{\log\log{x}}\Bigr)-\epsilon+c\frac{(\log{x})(\log{x/q})}{x/q}\Bigr)\frac{x}{\phi(q)\log{x}}.\]
\end{enumerate}
Here $c>0$ is an absolute constant.
\end{theorem}
 We do not claim Theorem \ref{theorem:Prime} is new; all the statements of Theorem \ref{theorem:Prime} follow immediately from the work \cite{Maier}, \cite{FriedlanderGranvilleIII} and  \cite{MaynardII}.

When the number of integers in the short interval or arithmetic progression is large compared with $\log{x}\log\log{x}$, the first terms in parentheses dominate the right hand sides, and we obtain the results of Maier \cite{Maier} and Friedlander and Granville \cite{FriedlanderGranvilleIII} that there are many intervals  and arithmetic progressions which contain more than the average number of primes by a constant factor. When the number of integers is small compared with $\log{x}\log\log{x}$, the final terms dominate, and we see that there are many intervals and arithmetic progressions which contain a number of primes which is a higher order of magnitude than the average number.

Part 3 of Theorem \ref{theorem:Prime} is one of the key ingredients in recent work of the author on large gaps between primes \cite{MaynardLargeGaps}. The previous best lower bound for the largest gaps between consecutive primes bounded by $x$ was due to Pintz \cite{Pintz} who explicitly indicated that one obstruction to improving his work was the inability of previous techniques to show the existence of many primes in such a short arithmetic progression with prime modulus (this issue was also clear in the key work of Maier and Pomerance \cite{MaierPomerance} on this problem). The independent improvement on this bound by Ford, Green, Konyagin and Tao \cite{FGKT} used the Green-Tao technology to show there are many primes $q$ for which there were many primes in an arithmetic progression modulo $q$, although it appears this approach does not give improvements to Theorem \ref{theorem:Prime}.

The aim of this paper is to establish an analogous statement to Theorem \ref{theorem:Prime} for the set $\mathcal{S}$ of numbers representable as the sum of two squares. We first require some notation. We define the counting functions $S(x)$ and $S(x;q,a)$, the constant $\mathfrak{S}$ and the multiplicative function $\phi_{\mathcal{S}}$ by
\begin{align}
S(x)&=\#\{n\le x:n\in\mathcal{S}\},\\
S(x;q,a)&=\#\{n\le x:n\in\mathcal{S},n\equiv a\pmod{q}\},\\
\mathfrak{S}&=\frac{1}{\sqrt{2}}\prod_{p\equiv 3\pmod{4}}\Bigl(1-\frac{1}{p^2}\Bigr)^{-1/2},\\
\phi_\mathcal{S}(p^e)&=
\begin{cases}
p^e,\quad &p\equiv 1\pmod{4},\\
p^{e+1}/(p+1),& p\equiv 3\pmod{4},\\
2^{e-1},&p=2\text{ and }e\ge 2,\\
2, &p^e=2.
\end{cases}
\end{align} 
With this notation, the average number $S(x+y)-S(x)$ of elements of $\mathcal{S}$ in a short interval $[x,x+y]$ is $\sim\mathfrak{S}y/\sqrt{\log{x}}$ and the average number $S(x;q,a)$ of elements of $\mathcal{S}$ in the arithmetic progression $a\pmod{q}$ for $(a,q)=1$ and $a\equiv1\pmod{(4,q)}$ which are less than $x$ is $\sim\mathfrak{S}x/(\phi_{\mathcal{S}}(q)\sqrt{\log{x}})$.

Finally, we let $F$ be the sieve function occurring in the upper bound half dimensional sieve, that is the function defined by the delay-differential equations
\begin{equation}
\begin{aligned}
F(s)&=\frac{2e^{\gamma/2}}{\pi^{1/2}s^{1/2}}\text{ for }0\le s\le 2,\qquad &(s^{1/2}F(s))'&=\frac{1}{2}s^{-1/2}f(s-1)\text{ for }s>2,\\
f(s)&=0\text{ for }0\le s\le 1,\qquad &(s^{1/2}f(s))'&=\frac{1}{2}s^{-1/2}F(s-1)\text{ for }s>1.
\end{aligned}
\end{equation}
It is known that $1<F(s)<1+O(e^{-s})$ for all $s>0$. 
\begin{theorem}\label{theorem:Squares}
Fix $\epsilon>0$, and let $\mathcal{S}$ denote the set of integers representable as the sum of two squares. There is a constant $c_0(\epsilon)>0$ such that we have the following:
\begin{enumerate}
\item Let $X,y$ satisfy $c_0(\epsilon)\le y\le X$. Then there exists at least $X^{1-1/\log\log{X}}$ values $x\in[X,2X]$ such that
%
\[S(x+y)-S(x)
\ge\Bigl(F\Bigl(\frac{\log{y}}{\log\log{x}}\Bigr)-\epsilon+c\frac{(\log{x})^{1/2}}{y^{9/10}}\Bigr)\frac{\mathfrak{S} y }{(\log{x})^{1/2}}.\]
\item Fix $a\in\mathcal{S}$, and let $Q,x$ satisfy $c_0(\epsilon)\le Q\le x/2$. Then there exists at least $Q^{1-1/\log\log{Q}}$ values $q\in [Q,2Q]$ such that
%
\[S(x;q,a)
\ge \Bigl(F\Bigl(\frac{\log{x/q}}{\log\log{x}}\Bigr)-\epsilon+c\frac{(\log{x})^{1/2}}{(x/q)^{9/10}}\Bigr)\frac{\mathfrak{S} x}{\phi_\mathcal{S}(q)(\log{x})^{1/2}}.\]
\item Let $q,x$ satisfy $c_0(\epsilon)\le q\le x/2$, and let $q$ have no prime factors less than $\log{x}/\log\log{x}$. Then there exists at least $q^{1-1/\log\log{q}}$ integers $a\in[1,q]$ such that
%
\[S(x;q,a)
\ge \Bigl(F\Bigl(\frac{\log{x/q}}{\log\log{x}}\Bigr)-\epsilon+c\frac{(\log{x})^{1/2}}{(x/q)^{9/10}}\Bigr)\frac{\mathfrak{S} x}{\phi_{\mathcal{S}}(q)(\log{x})^{1/2}}.\]
\end{enumerate}
Here $c>0$ is an absolute constant.
\end{theorem}
Similarly to Theorem \ref{theorem:Prime}, when the number of integers in the short interval or arithmetic progression is large compared with $(\log{x})^{5/9}$, we find that there more than the average number of integers representable as the sum of two square, which is a result of Balog and Wooley \cite{BalogWooley}, based on Maier's method. When the number of integers is small compared with $(\log{x})^{5/9}$, then the number of integers representable as the sum of two squares is of a higher order of magnitude than the average. We note that unlike Theorem \ref{theorem:Prime}, Theorem \ref{theorem:Squares} improves on the Maier matrix bounds in a range beyond the trivial region by a full positive power of $\log{x}$.

The constant $9/10$ appearing in the exponent in the denominators of the final terms of Theorem \ref{theorem:Squares} is certainly not optimal; with slightly more effort one could certainly improve this. We satisfy ourselves with a weaker bound for simplicity here, the key point of interest being one can obtain a constant less than 1, and so there are intervals of length considerably larger than the average gap between elements of $\mathcal{S}$ which contain a higher order of magnitude than the average number of elements.

Analogously to the work of Granville and Soundararajan \cite{GranvilleSound}, we believe the phenomenon of significant failures of equidistribution in short intervals should hold in rather more general `arithmetic sequences' which are strongly equidistributed in arithmetic progressions (in the sense of a weak Bombieri-Vinogradov type theorem). We hope to return to this in a future paper.
\section{Overview}
We give a rough overview of the methods, emphasizing the similarities between Maier's matrix method and the GPY method. For simplicity we concentrate on the case when we are looking for primes in short intervals; the arguments for arithmetic progressions are similar, and when looking for sums of two squares one wishes to `sieve out' only primes congruent to $3\pmod{4}$.

We estimate a weighted average of the number of primes in a short interval, and wish to show that this weighted average is larger than what one would obtain if the primes were evenly distributed. More specifically, we consider the ratio
\begin{equation}
\Bigl(\sum_{X<n<2X}\#\Bigl\{p\text{ prime}\in[n,n+h]\Bigr\}w_n\Bigr)\Bigg /\Bigl(\sum_{X<n<2X}\frac{h}{\log{X}}w_n\Bigr),\label{eq:Ratio}
\end{equation}
where $w_n$ are some non-negative weights chosen to be large when the interval $[n,n+h]$ contains many numbers with no small prime factors, and small otherwise.

If the ratio \eqref{eq:Ratio} is greater than some constant $c_0$, then it is clear that there must be at least one $n\in[X,2X]$ such that $[n,n+h]$ contains at least $c_0h/\log{x}$ primes (i.e. $c_0$ times the average number). Thus we aim to choose $w_n$ to make this ratio as large as possible.

In the range $h=o(\log{X})$, we choose the weights $w_n$ to be concentrated on integers $n$ such that the interval $[n,n+h]$ \textit{only}\footnote{We ignore the elements of $[n,n+h]$ which are a multiple of some prime $p\le h$, which is unavoidable.} contains numbers with no small prime factors
\footnote{Here a `small prime' refers to one bounded by roughly $\exp(h^{-1+o(1)}\log{x})$.}
. To achieve this we choose $w_n$ to mimic weights of Selberg's $\Lambda^2$-sieve
\begin{equation}
w_n\approx\Bigl(\sum_{\substack{d_1,\dots,d_h\\ d_i|n+i}}\lambda_{d_1,\dots,d_h}\Bigr)^2,\label{eq:SelbergWeights}
\end{equation}
where $\lambda_{d_1,\dots,d_h}$ are real numbers chosen later to optimize the final result.

In the range $h\approx (\log{x})^\lambda$ for some constant $\lambda>1$ it is unavoidable that a large number of elements of $[n,n+h]$ will have small prime factors, and so we require a slightly different choice of the weights $w_n$. Instead we observe that if $n$ is a multiple of primes less than $z=(\log{x})^{\lambda'}$ (where $\lambda'<1$), then the probability that a random element of $[n,n+h]$ is coprime to all these primes fluctuates depending on the size of $\lambda/\lambda'$, and can be more than the average for general $n$. More specifically, we take 
\begin{equation}
w_n\approx
\begin{cases}
1,\qquad &n\equiv 0\pmod{\prod_{p<z}p},\\
0,&\text{otherwise.}
\end{cases}\label{eq:MaierWeights}
\end{equation}
The number of elements of $[n,n+h]$ which have no prime factors less than $z$ is given by
\begin{equation}
\#\{j\in[1,h]:(j,p)=1\forall p< z\}\sim e^\gamma \omega\Bigl(\frac{\log{h}}{\log{z}}\Bigr)h\prod_{p<z}\frac{p-1}{p}\label{eq:RoughCount}
\end{equation}
for $z\rightarrow \infty$ and $z<h^{1-\epsilon}$ (here $\omega$ is the Buchstab function). Thus, given $h$ we can choose $z$ so that the number of elements coprime to small primes is larger than the expected number in a random interval by a factor $e^\gamma \omega(\log{h}/\log{z})>1$.

Given this choice of weights, to estimate the numerator in \eqref{eq:Ratio} we rewrite the count of the primes as a sum of the indicator function of the primes, and swap the order of summation. This gives
\begin{equation}
\sum_{j=1}^h\sum_{X<n\le 2X}\mathbf{1}_{\mathbb{P}}(n+j)w_n.
\end{equation}
(Here, and throughout this paper, $\mathbf{1}_{\mathcal{A}}$ will denote the indicator function of a set $\mathcal{A}$, in this case the set $\mathbb{P}$ of primes.) Our choice of $w_n$ means that we can estimate this expression using knowledge of the distribution of primes in arithmetic progressions\footnote{One needs to be slightly careful about the possible effect of Siegel zeros here, but this is a minor technical issue.}. The inner sum vanishes or has a simple asymptotic expression depending on the arithmetic structure of $j$. In the case of the Maier weights, this leads us to \eqref{eq:RoughCount}, which gives our asymptotic for the numerator, and hence the ratio \eqref{eq:Ratio}. In the range $h=o(\log{X})$, we obtain a quadratic form in the coefficients $\lambda_{d_1,\dots,d_h}$ for the numerator, and similarly a quadratic form for the denominator. We then make a choice of these $\lambda_{d_1,\dots,d_h}$ which (approximately) maximizes the ratio of these quadratic forms.

\begin{remark}
One can heuristically investigate Selberg-sieve weights \eqref{eq:SelbergWeights} for larger $h$, and (assuming various error terms are negligble) conclude that an approximately optimal choice of weights $\lambda_{d_1,\dots,d_k}$ should be given by $\lambda_{d_1,\dots,d_k}\approx \mu(d_1\dots d_k)$ if $\prod_{i=1}^k d_i$ has all of its prime factors less than a small multiple of $\log{x}$ (and $0$ otherwise). This corresponds precisely to Maier's choice of weights \eqref{eq:MaierWeights}.
\end{remark}
\section{Admissible sets of linear functions}
The bounds involving the sieve function $F$ in Theorem \ref{theorem:Squares} follow from the argument of Balog and Wooley \cite{BalogWooley}, and by making minor adaptions analogous to the work of Friedlander-Granville \cite{FriedlanderGranvilleIII}. We sketch such arguments in Section \ref{section:Irregularities}.

 Therefore the main task of this paper is to establish the bounds coming from the final terms in parentheses of Theorem \ref{theorem:Squares}. We will do this by an adaptation of the GPY sieve method. Indeed, we will actually prove a rather stronger result, given by Theorem \ref{theorem:Admiss} below, which is roughly analogous to the main theorem of \cite{MaynardII}.

To ease notation, we let $\langle P_1\rangle$ denote the set of integers composed only of primes congruent to $1\pmod{4}$. Similarly let $\langle P_3\rangle$ those composed of primes congruent to $3\pmod{4}$.
\begin{definition}[$\langle P_3\rangle$ - Admissibility]
We say a set $\mathcal{L}=\{L_1,\dots,L_k\}$ of distinct linear functions $L_i(n)=a_i n+b_i$ with integer coefficients is \emph{$\langle P_3\rangle$-admissible} if for every prime $p\equiv 3\pmod{4}$ there is an integer $n_p$ such that $(\prod_{i=1}^k L_i(n_p),p)=1$, and if $a_i,b_i>0$ for all $i$.
\end{definition}
\begin{theorem}\label{theorem:Admiss}
Let $x$ be sufficiently large and $k\le (\log{x})^{1/5}$. Then there is a prime $p_0\gg \log\log{x}$ such that the following holds.

Let $\mathcal{L}=\{L_1,\dots,L_k\}$ be a $\langle P_3\rangle$-admissible set of linear functions such that the coefficients of $L_i(n)=a_in+b_i$ satisfy $0<a_i\le (\log{x})^{1/3}$, $(2p_0,a_i)=1$,  and $0<b_i<x$. There exists an absolute constant $C>0$ such that
\[\#\{n\in[x,2x]:\text{ at least $k^{1/2}/C$ of $L_1(n),\dots,L_k(n)$ are in $\mathcal{S}$}\} \ge x\exp(-(\log{x})^{1/2}).\]
\end{theorem}
We remark that the conditions $k\le (\log{X})^{1/5}$, $2\nmid a_i$, $a_i<(\log{X})^{1/3}$ and $b_i\le x$ are considerably weaker than what our method requires, but simplify some of the later arguments slightly and are sufficient for our application.

In the region where $y$, $x/Q$ and $x/q$ are small compared with $(\log{x})^{5/9}$, the bounds in Theorem \ref{theorem:Squares} follow easily from Theorem \ref{theorem:Admiss}. For sums of squares in short intervals, we consider $y\le (\log{x})^{5/9+\epsilon}$. We take $k=y^{1/5}$, and let $\mathcal{L}=\{L_1,\dots,L_k\}$ be the $\langle P_3\rangle$-admissible set of linear functions $L_i(n)=n+h_i$, where $h_1,\dots,h_k$ are the first $k$ positive integers in $\langle P_1\rangle$ which are coprime with $p_0$. (This is $\langle P_3\rangle$-admissible, since $L_i(0)$ is coprime to all primes $p\equiv 3\pmod{4}$ for all $i$). We note that $h_k\ll k(\log{k})^{1/2}\le y$. It then follows immediately from Theorem \ref{theorem:Admiss} that there are $\gg X\exp(-(\log{X})^{1/2})$ values of $n\in[X,2X]$ such that $\gg k^{1/2}=y^{1/10}$ of the $L_i(n)$ are in $\mathcal{S}$, and hence $\gg X^{1-1/\log\log{X}}$ values of $x\in[X,2X]$ such that $[x,x+y]$ contains at least $y^{1/10}$ elements of $\mathcal{S}$.

Parts $2$ and $3$ of Theorem \ref{theorem:Squares} follow analogously, defining $k=(x/Q)^{1/5}$ and $L_i(n)=a+h_i n$ when $Q>x(\log{x})^{-5/9-\epsilon}$ for part $2$, and defining $k=(x/q)^{1/10}$ and $L_i(n)=n+h_i q$ when $q>x(\log{x})^{-5/9-\epsilon}$ for part $3$.
We note that it follows from Theorem \ref{theorem:Squares} that actually there are many intervals $[x,x+y]$ containing $\gg y^{1/2}(\log{y})^{-1/4}$ elements of $\mathcal{S}$ (rather than $\gg y^{1/10}$), for the smaller range $1\le y\le (\log{x})^{1/5}$.
\section{Irregularities from the Maier matrix method}\label{section:Irregularities}
As previously mentioned, the bounds in Theorem \ref{theorem:Squares} when $y$, $x/Q$ or $x/q$ large compared with $(\log{x})^{5/9}$ follow from minor adaptions of the argument of Balog and Wooley \cite{BalogWooley}. In particular, the relevant bound for part $2$ of Theorem \ref{theorem:Squares} follows from the proof of \cite[Theorem 1]{BalogWooley}. We give  brief outline of the argument for parts $2$ and $3$, leaving the details to the interested reader.

We first consider part $2$ of Theorem \ref{theorem:Squares}. We assume that $(\log{x})^{5/9-\epsilon}\le x/Q\le (\log{X})^{O(1)}$ since otherwise either the result follows from Theorem \ref{theorem:Admiss}, or the result is trivial. For simplicity we shall assume that $a$ is odd; the case for even $a$ is entirely analogous. We let $P=\prod_{p\equiv3\pmod{4},p\le z}p^{\alpha_p}$, where $\alpha_p$ is the least odd integer such that $p^{\alpha_p}\ge a(4x/Q+1)$ and $z=\log{x}/(\log\log{x})^2$. We note that $P=Q^{o(1/\log\log{Q})}$. Finally, we let $\mathcal{Q}=\{q\in[(1-\delta)Q+4a,Q+4a]:q\equiv 4a\pmod{4P}\}$ for some small fixed constant $\delta=\delta(\epsilon)>0$. We have
\begin{align}
\sum_{q\in \mathcal{Q}}
S(x;q,a)&
\ge \sum_{r\le x/Q-1}\,\sum_{q\in\mathcal{Q}}\mathbf{1}_{\mathcal{S}}(a+r q)\nonumber\\
&= \sum_{r\le x/Q-1}\,\sum_{m\in[(1-\delta)Q/4P,Q/4P]}\mathbf{1}_{\mathcal{S}}(a(4r+1)+4mrP).
\end{align}
The inner sum is counting elements of $\mathcal{S}$ in an arithmetic progression to modulus $4 r P$. We note that by construction of $P$ we must have $(P,a)=e^2$ for some $e\in\langle P_3 \rangle$ since $a\in\mathcal{S}$, and similarly if the inner sum is non-empty we must have $(P,a(4r+1))=e^2 d^2$ since $a(4r+1)+4mrP\in \mathcal{S}$. Moreover $P/(e^2d^2)$ and $P$ are composed of the same prime factors, since each prime occurs in $P$ with odd multiplicity. Thus, if $r$ is such that $(P,4r+1)=d^2$, the inner sum is
\begin{equation}
(1+o(1))\frac{\mathfrak{S}\delta r Q/(e^2d^2)}{\phi_\mathcal{S}(4rP/(d^2e^2))\sqrt{\log{rQ/d^2e^2}}}\sim \frac{\mathfrak{S}\delta Q}{2\phi_{\mathcal{S}}(P)\sqrt{\log{x}}}.
\end{equation}
Therefore, letting $u d^2=4r+1$, we obtain
\begin{equation}
\sum_{q\in \mathcal{Q}}
S(x;q,a)
\ge \frac{(1+o(1))\mathfrak{S}\delta Q}{2\phi_{\mathcal{S}}(P)\sqrt{\log{x}}}\sum_{d^2|P}\,\sum_{\substack{u<(4x/Q+1)/d^2\\ u\equiv 1\pmod{4}\\ (u,P)=1}}1.
\end{equation}
This double sum is exactly the sum $\mathcal{R}^+$ which is estimated in \cite[proof of Lemma 4.3]{BalogWooley}. In particular, they show that
\begin{equation}
\sum_{d^2|P}\,\sum_{\substack{u<(4x/Q+1)/d^2\\ u\equiv 1\pmod{4}\\ (u,P)=1}}1\sim \frac{\phi_\mathcal{S}(P)}{P}F\Bigl(\frac{\log{x/Q}}{\log{z}}\Bigr).\label{eq:R+Bound}
\end{equation}
Recalling that $z=\log{x}/(\log\log{x})^2$, one arrives at (for $\delta$ sufficiently small)
\begin{equation}
\sum_{q\in \mathcal{Q}}\Biggl(
S(x;q,a)-\Bigl(F\Bigl(\frac{\log{x/q}}{\log\log{x}}\Bigr)-\epsilon\Bigr)\frac{\mathfrak{S}x}{\phi_\mathcal{S}(q)\sqrt{\log{x}}}\Biggr)\gg \frac{x}{P\sqrt{\log{x}}}.
\end{equation}
This implies the relevant bound in part $2$ of Theorem \ref{theorem:Squares}.

We now consider part $3$ of Theorem \ref{theorem:Squares}. We let $\tilde{P}=\prod_{p\equiv 3\pmod{4},p\le z}p^{\beta_p}$ where $\beta_p$ is the least odd integer such that $p^{\beta_p}\ge 4x/q$, and let $\mathcal{A}=\{a\in[(1-\delta)q,q]:a\equiv a_0\pmod{P}\}$. We see that
\begin{align}
\sum_{a\in \mathcal{A}}
S(x;q,a)
&\ge \sum_{r\le (x-q)/q}\,\sum_{a\in\mathcal{A}}\mathbf{1}_{\mathcal{S}}(a+r q)\nonumber\\
&= \sum_{r\le x/q-1}\,\sum_{m\in[((1-\delta)q-a_0)/P,(q-a_0)/P]}\mathbf{1}_{\mathcal{S}}(a_0+rq+mP).
\end{align}
We see the inner sum counts elements of $\mathcal{S}$ in an arithmetic progression modulo $P$. We choose $a_0\equiv q/4\pmod{P}$, so that $(a_0+rq,P)=(4r+1,P)$ since, by assumption, $q$ has no prime factors in common with $P$. Thus $(4r+1,P)=d^2$ for some $d\in\langle P_3\rangle$ by the same argument as above, and we obtain
\begin{align}
\sum_{a\in \mathcal{A}}
S(x;q,a)
&\ge \frac{(1+o(1))\mathfrak{S}\delta q}{\phi_{\mathcal{S}}(P)\sqrt{\log{x}}}\sum_{d^2|P}\,\sum_{\substack{u<(4x/q-3)/d^2 \\ u\equiv 1\pmod{4}\\ (u,P)=1}}1.
\end{align}
Again, using \eqref{eq:R+Bound}, we obtain
\begin{equation}
\sum_{a\in \mathcal{A}}\Biggl(
S(x;q,a)
-\Bigl(F\Bigl(\frac{\log{x/q}}{\log\log{x}}\Bigr)-\epsilon\Bigr)\frac{\mathfrak{S}x}{\phi_\mathcal{S}(q)\sqrt{\log{x}}}\Biggr)\gg \frac{x}{P\sqrt{\log{x}}},
\end{equation}
which implies the relevant bound for part $3$.
\section{Setup for Theorem \ref{theorem:Admiss}}
The improvements of the GPY sieve in the author's work \cite{Maynard,MaynardII} are not significant for the application of finding sums of two squares in short intervals (the improvement would be $y^{o(1)}$ in the final term for part $1$ of Theorem \ref{theorem:Squares}, for example), and so we will use a uniform version of the simplest GPY sieve.

In order to get a result that applies in the larger range, however, it is necessary to modify the sieve to the application; a uniform version of the argument in \cite{GGPY} (i.e. counting numbers with two prime factors both congruent to $1\pmod{4}$) would only give non-trivial results in the region $y\ll (\log{x})^{1/2}\log\log{x}$, for example.

We follow a similar argument to \cite{GGPY} (and attempt to keep the notation similar), but make modifications to allow for uniformity (similar to those in \cite{GPYII} and \cite{MaynardII}, although it is simpler in this context) and to specialize so as to remove only primes congruent to $3\pmod{4}$ (so we will apply sieves of `dimension' $k/2$ and $(k+1)/2$ instead of dimension $k$ and $k+1$).

In order to state our setup, we require the following Lemma.
\begin{lemma}\label{lemma:BV}
Let $x>10$ and $\epsilon>0$. There exists a prime $p_0\in[(\log\log{x})/2,x]$ such that
\[\sum_{\substack{q<x^{1/2-\epsilon}\\ (q,p_0)=1}}\,\sup_{\substack{(a,q)=1\\ x'\le x}}\Bigl|\pi(x';q,a)-\frac{\pi(x')}{\phi(q)}\Bigr|\ll_\epsilon x\exp(-c_0\sqrt{\log{x}}),\]
for some absolute constant $c_0>0$.
\end{lemma}
\begin{proof}
This follows from known results on the repulsion of zeros of $L$-functions. Following \cite[Chapter 28]{Davenport}, we have (for a suitable constant $c>0$)
\begin{align}
\sum_{\substack{q<x^{1/2-\epsilon}\\ (q,p_0)=1}}\sup_{\substack{(a,q)=1\\ x'\le x}}\Bigl|\pi(x';q,a)-&\frac{\pi(x')}{\phi(q)}\Bigr|\ll x\exp(-c\sqrt{\log{x}})\nonumber\\
&+\log{x}\sum_{\substack{q<\exp(2c\sqrt{\log{x}})\\ (q,p_0)=1}}\;\sup_{x'\le x}\;\;\sideset{}{^*}\sum_{\chi\pmod{q}}\frac{|\psi'(x',\chi)|}{\phi(q)}.
\end{align}
However, by the Landau-Page theorem \cite[Chapter 20]{Davenport}, we have that if $c$ is sufficiently small then
\begin{equation}
\phi(q)^{-1}\sideset{}{^*}\sum_\chi|\psi'(x',\chi)|\ll x\exp(-3c\sqrt{\log{x}})
\end{equation}
for all $q<\exp(2c\sqrt{\log{x}})$ and $x'\le x$, except possibly those which are a multiple of an exceptional modulus $q_1$. Such an exceptional modulus must satisfy $\log{x}\ll q_1(\log{q_1})^4\ll x$, and must be square-free apart from a factor of at most 4. Taking $p_0$ to be the largest prime factor of $q_1$ then gives the result.
\end{proof}
Let $\mathcal{L}=\{L_1,\dots,L_k\}$ be the $\langle P_3\rangle$-admissible set of Theorem \ref{theorem:Admiss}, $p_0$ be the prime given by Lemma \ref{lemma:BV}, and let
\begin{equation}
W=\prod_{\substack{p\le 2(\log{X})^{1/3}\\p\equiv 3\pmod{4}\\p\ne p_0}}p.
\end{equation}
It will be convenient to choose $w_n$ to be 0 unless $n\equiv v_0\pmod{W}$, where $v_0$ is chosen such that $(L_j(v_0),W)=1$ for each $j$. (Such a $v_0$ exists by the $\langle P_3\rangle$-admissibility of $\{L_1,\dots,L_k\}$ and the Chinese Remainder Theorem.) This allows us to ignore $n$ for which one of the $L_i(n)$ has a small prime factor congruent to $3\pmod{4}$.

We then define our weights $w_n$ by
\begin{equation}
w_n=
\begin{cases}
\Bigl(\sum_{d|\prod_{i=1}^k L_i(n)}\lambda_d\Bigr)^2,\quad &\text{if }n\equiv v_0\pmod{W},\\
0,&\text{otherwise,}
\end{cases}
\end{equation}
for some coefficients $\lambda_d$ (which we will choose explicitly later) satisfying
\begin{equation}
\lambda_d=0\quad \text{if $d>X^{1/10}$ or $\mu^2(d)=0$ or $d\notin\langle P_3\rangle$ or $(d,p_0W)\ne 1$}.
\end{equation}
We recall that $\langle P_1\rangle$ and $\langle P_3\rangle$ denote the sets of integers composed only of prime factors $p\equiv 1\pmod {4}$ and $p\equiv 3\pmod{4}$ respectively. Finally, we define the function $\nu=\nu_{\mathcal{L}}$ on primes by
\begin{equation}
\nu(p)=\#\{1\le n<p:\prod_{i=1}^k L_i(n)\equiv 0\pmod{p}\}.
\end{equation}
We note that since $\{L_1,\dots,L_k\}$ is $\langle P_3\rangle$-admissible, $\nu(p)<p$ for any prime $p\equiv 3\pmod{4}$.
\section{Proof of Theorem \ref{theorem:Admiss}}
\begin{lemma}\label{lemma:Const}
Assume the hypotheses of Theorem \ref{theorem:Admiss} and that $\sup_{d}|\lambda_d|\ll (\log{X})^{k/2}$. Then
\[\sum_{\substack{X< n\le 2X\\n\equiv v_0\pmod{W}}}w_n=\frac{X}{W}\sum_{d,e}\lambda_d\lambda_e\prod_{p|de}\frac{\nu(p)}{p}+O\Bigl(X^{1/5+\epsilon}\Bigr).\]
\end{lemma}
\begin{proof}
From the definition of $w_n$, we have
\begin{equation}
\sum_{\substack{X< n\le 2X\\n\equiv v_0\pmod{W}}}w_n=\sum_{d,e}\lambda_d\lambda_e\sum_{\substack{X< n\le 2X\\n\equiv v_0\pmod{W}\\ d,e|\prod_{i=1}^k L_i(n)}}1.
\end{equation}
We concentrate on the inner sum. For each prime $p_1|de$, we must have $L_i(n)\equiv 0\pmod{p_1}$ for some $1\le i\le k$, and so $n$ can lie in one of $\nu(p_1)$ residue classes. Thus, using the Chinese remainder theorem, we can rewrite the inner sum as a sum of $\prod_{p|de}\nu(p)$ sums of $n$ in a single residue class $\pmod{W[d,e]}$, where $[d,e]$ is the least common multiple of $d$ and $e$ (since $\lambda_d=0$ if $(d,W)\ne 1$ or $d$ is not square-free). The number of integers $X<n\le 2X$ counted in any such residue class is then $X/([d,e]W)+O(1)$. Putting this together gives
\begin{align}
\sum_{d,e}\lambda_d\lambda_e\sum_{\substack{X< n\le 2X\\n\equiv v_0\pmod{W}\\ d,e|\prod_{i=1}^k L_i(n)}}1&=\sum_{d,e}\lambda_d\lambda_e \Bigl(\prod_{p|de}\nu(p)\Bigr)\Bigl(\frac{X}{W[d,e]}+O(1)\Bigr).
\end{align}
We see that the main term gives the corresponding main term in the statement of the lemma. Since $\nu(p)\le k$, the contribution from the error term is
\begin{align}
\ll \sup_{d}|\lambda_d|^2\Bigl(\sum_{d<X^{1/10}}\prod_{p|d}k\Bigr)^2&\le \sup_{d}|\lambda_d|^2\Bigl(X^{1/10}\sum_{d<X}\frac{\prod_{p|d}k}{d}\Bigr)^2\nonumber\\
&\ll \sup_{d}|\lambda_d|^2X^{1/5}(1+\log{X})^{2k}.
\end{align}
By assumption of the Lemma, $k\le (\log{X})^{1/3}$ and $\sup_{d}|\lambda_d|\ll (\log{X})^{k/2}$, so this contribution is $O(X^{1/5+\epsilon})$.
\end{proof}
\begin{lemma}\label{lemma:Squares}
Assume the hypotheses of Theorem \ref{theorem:Admiss} and that $\sup_{d}|\lambda_d|\ll (\log{X})^{k/2}$. Then
\begin{align*}
\sum_{\substack{X< n\le 2X\\n\equiv v_0\pmod{W}}}\#\Bigl\{i:L_i(n)\in\mathcal{S}\Bigr\}w_n&\gg \frac{X k}{\phi(W)(\log{X})^{1/2}}\sum_{d,e}\lambda_d\lambda_e\prod_{p|de}\frac{\nu(p)-1}{p-1}\\
&\qquad+O(X\exp(-c\sqrt{\log{X}}))
\end{align*}
for some absolute constant $c>0$.
\end{lemma}
\begin{proof}
%
We first obtain a lower bound by only counting a subset of $\#\{1\le i\le k:L_i(n)\in\mathcal{S}\}$ which will be more convenient for our manipulations (we lose an unimportant constant factor in doing this). Let $\mathcal{S}'\subseteq\mathcal{S}$ be given by
\begin{equation}
\mathcal{S}'=\{n\in\langle P_1\rangle:n=mp\text{ for some $m<X^{1/3}$ and some  prime $p$}\}.
\end{equation} 
We obtain a lower bound by only counting elements of $\mathcal{S}'$. This gives
\begin{align}
\sum_{\substack{X< n\le 2X\\n\equiv v_0\pmod{W}}}&\#\Bigl\{i:L_i(n)\in\mathcal{S}\Bigr\}w_n
\ge \sum_{\substack{X< n\le 2X\\n\equiv v_0\pmod{W}}}\#\Bigl\{1\le i\le k:L_i(n)\in\mathcal{S}'\Bigr\}w_n\nonumber\\
&\qquad\qquad=\sum_{j=1}^k\sum_{\substack{X< n\le 2X\\n\equiv v_0\pmod{W}}}\mathbf{1}_{\mathcal{S}'}(L_j(n))\Bigl(\sum_{d|\prod_{i=1}^k L_i(n)}\lambda_d\Bigr)^2.
\end{align}
By definition, if $L_j(n)=a_jn+b_j\in\mathcal{S}'$, then $L_j(n)$ can be written uniquely as $mp$ for some $m<X^{1/3}$ composed only of primes congruent to $1\pmod{4}$ and with $(m,a_j)=1$, and some prime $p\equiv 1\pmod{4}$. Making this substitution and rearranging the sums, we are left to estimate
\begin{align}
\sum_{j=1}^k\sum_{\substack{m\le X^{1/3}\\ m\in \langle P_1\rangle \\ (m,a_j)=1}}\sum_{\substack{d,e<X^{1/10}\\ d,e\in \langle P_3\rangle}}\lambda_d\lambda_e\sum_{\substack{(a_jX+b_j)/m<p<(2a_jX+b_j)/m\\ p\equiv 1\pmod{4}\\ p\equiv \overline{m}(a_jv_0+b_j)\pmod{a_jW}\\ [d,e]|\prod_{i=1}^k L_i((pm-b_j)/a_j)}}1.
\end{align}
For each prime $p_1|de$ with $p_1>(\log{x})^{1/3}$, there are $\nu(p)-1$ possible primitive residue classes for $p\pmod{p_1}$ such that $\prod_{i=1}^k L_i((mp-b_j)/a_j)\equiv 0\pmod{p_1}$. (All $\nu(p)$ residue classes for which $\prod_{i=1}^k L_i(n)\equiv 0\pmod{p_1}$ correspond to a primitive residue class for $p$ except for $n\equiv -\overline{a_j}b_j\pmod{p_1}$, and these all exist since $a_i<(\log{X})^{1/3}<p_1$ and $p_1\equiv 3\pmod{4}$ so $p_1\nmid m$.) Thus, by the Chinese remainder theorem, we can rewrite the inner sum as a sum of $\prod_{p|de}(\nu(p)-1)$ different sums of primes in arithmetic progressions to modulus $4 a_j W\prod_{p|de}p$ (since $2\nmid a_j$). Thus the inner sum is
\begin{align}
&\frac{1}{\phi(4 a_j W)}\Bigl(\pi\Bigl(\frac{2a_j X+b_j}{m}\Bigr)-\pi\Bigl(\frac{a_j X+b_j}{m}\Bigr)\Bigr)\prod_{p|de}\frac{\nu(p)-1}{p-1}\nonumber\\
&\qquad\qquad+O\Bigl(E\Bigl(\frac{3 a_j X}{m};4 a_j W[d,e]\Bigr)\prod_{p|de}k\Bigr),\label{eq:PrimesInAP}
\end{align}
where
\begin{equation}
E(x;q)=\sup_{\substack{(a,q)=1\\ x'\le x}}\Bigl|\pi(x';q,a)-\frac{\pi(x')}{\phi(q)}\Bigr|,\qquad [d,e]=\prod_{p|de}p.
\end{equation}
We first consider the error term of \eqref{eq:PrimesInAP}. Letting $r=4a_j W[d,e]\ll X^{1/4}$ (noting that $(r,p_0)=1$) and $\lambda_{max}=\sup_{d}|\lambda_d|$, this contributes at most
\begin{align}
&k \lambda_{\text{max}}^2\sum_{m<X^{1/3}}\sum_{\substack{r\ll X^{1/4}\\ (r,p_0)=1}}E\Bigl(\frac{3 a_j X}{m};r\Bigr)\prod_{p|r}3k\nonumber\\
&\le k\lambda_{\text{max}}^2\sum_{m<X^{1/3}}\Bigl(\frac{3 a_j X}{m}\sum_{r<4 W X^{1/5}}\frac{\prod_{p|r}9k^2}{r}\Bigr)^{1/2}\Bigl(\sum_{\substack{r<4 W X^{1/5}\\ (r,p_0)=1}}E\Bigl(\frac{3 a_j X}{m};r\Bigr)\Bigr)^{1/2}\nonumber\\
&\ll k \lambda_{max}^2 \Bigl( a_j X(9k^2+\log{X})^{9k^2}\Bigr)^{1/2}\Bigl(a_jX\exp(-c_0\sqrt{\log{X^{2/3}}})\Bigr)^{1/2}\sum_{m<X^{1/3}}\frac{1}{m}.\label{eq:APError}
\end{align}
Here we have used Cauchy-Schwarz and the trivial bound $E(X,q)\ll X/q$ for $q<X$ in the first line, and Lemma \ref{lemma:BV} in the second. Thus, since $k\le (\log{X})^{1/5}$, $a_j\le (\log{X})^{1/3}$ and $\lambda_{\text{max}}\ll (\log{X})^{k/2}$ by assumption of the Lemma, the contribution from \eqref{eq:APError} is $O(X\exp(-\frac{c_0}{4}\sqrt{\log{X}}))$.

We now consider the contribution from the main term of \eqref{eq:PrimesInAP}. The main term factorizes, simplifying to
\begin{equation}
\frac{1}{\phi(4W)}\Bigl(\prod_{p|a_j,p\nmid W}\frac{p}{p-1}\Bigr)\Bigl(\sum_{\substack{m<X^{1/3}\\ m\in \langle P_1\rangle \\ (m,a_j)=1}}\frac{X+o(X)}{m\log{\frac{X}{m}}}\Bigr)\Bigl(\sum_{d,e}\lambda_d\lambda_e\prod_{p|de}\frac{\nu(p)-1}{p-1}\Bigr).
\end{equation}
The second term is straightforward to evaluate. For $t>X^\epsilon$ we have that
\begin{equation}
\sum_{\substack{m<t\\ m\in\langle P_1\rangle \\ (m,a_j)=1}}1\gg t(\log{t})^{-1/2} \prod_{\substack{p|a_j\\ p\equiv 1\pmod{4}}}\frac{p-1}{p},
\end{equation}
and so by partial summation we see that 
\begin{equation}
\sum_{\substack{m<X^{1/3}\\ m\in \langle P_1\rangle \\ (m,a_j)=1}}\frac{X+o(X)}{m\log{\frac{X}{m}}}\gg X(\log{X})^{-1/2} \prod_{\substack{p|a_j\\ p\equiv 1\pmod{4}}}\frac{p-1}{p}.\end{equation}
Noting that the products over prime divisors of $a_j$ cancel, this completes the proof of the Lemma.
\end{proof}
\begin{lemma}\label{lemma:Summation}
Fix $A>0$. Let $1/A\le \kappa\le (\log{R})^{1/3}$ and let $f:[0,1]\rightarrow \mathbb{R}$ be a smooth non-negative function. Let $\gamma(p)$ satisfy $0\le \gamma(p)\le \min(A\kappa,(1-1/A)p)$ for all $p<R$, and
\[-L\le \sum_{w<p<z}\frac{\gamma(p)\log{p}}{p}-\kappa \log{z/w}\le A\]
for all $2\le w\le z<R$. Then
\begin{align*}
\sum_{r<R}\mu^2(r)&\Bigl(\prod_{p|r}\frac{\gamma(p)}{p-\gamma(p)}\Bigr)f\Bigl(\frac{\log{r}}{\log{R}}\Bigr)\\
&=\mathfrak{S}_\gamma \frac{(\log{R})^{\kappa}}{\Gamma(\kappa)}\int_0^1f(t)t^{\kappa-1}dt\Bigl(1+O_{A,f}\Bigl(\frac{\kappa L+\kappa ^2\log\log{R}}{\log{R}}\Bigr)\Bigr),
\end{align*}
where
\[\mathfrak{S}_\gamma=\prod_{p<R}\Bigl(1-\frac{\gamma(p)}{p}\Bigr)^{-1}\Bigl(1-\frac{1}{p}\Bigr)^{\kappa}.\]
\end{lemma}
\begin{proof}
For $f=1$, this is a version of \cite[Lemma 5.4]{HalberstamRichert}, with the dependence on $\kappa$ kept explicit. In particular, this is obtained in the proof of \cite[Lemma 5.1]{FordPratt}. The case of $f\ne 1$ then follows immediately by partial summation.
\end{proof}

\begin{lemma}\label{lemma:Ratio}
Let $k\le (\log{X})^{1/5}$, and let $\lambda_d$ be defined by 
\[\lambda_d=\mu(d)\Bigl(\prod_{p|d}\frac{p}{\nu(p)}\Bigr)\sum_{d|r}y_r\Bigl(\prod_{p|r}\frac{\nu(p)}{p-\nu(p)}\Bigr),\]
for $(d,p_0W)=1$, where
\[ y_r=\begin{cases}\mu(r)^2,\qquad &\text{if }r\in\langle P_3\rangle,\, (r,W p_0)=1,\, r<X^{1/10},\\
0,&\text{otherwise.}\end{cases}\]
Then
\[\sum_{d,e}\lambda_d\lambda_e\prod_{p|de}\frac{\nu(p)-1}{p-1}\gg k^{-1/2}(\log{X})^{1/2}\frac{\phi(W)}{W}\sum_{d,e}\lambda_d\lambda_e\prod_{p|de}\frac{\nu(p)}{p}\gg 1.\]
\end{lemma}
\begin{proof}
We first note that the $y_r$ variables diagonalize the second quadratic form. Substituting the expression for $\lambda_d$ in terms of $y_r$ given by then Lemma, we obtain
\begin{align}
\sum_{d,e}&\lambda_d\lambda_e\prod_{p|de}\frac{\nu(p)}{p}=\sum_{r,s}y_ry_s\Bigl(\prod_{p|r}\frac{\nu(p)}{p-\nu(p)}\Bigr)\Bigl(\prod_{p|s}\frac{\nu(p)}{p-\nu(p)}\Bigr)\nonumber\\
&\qquad\qquad\times\sum_{d,e:d|r,e|s}\mu(d)\mu(e)\Bigl(\prod_{p|d}\frac{p}{\nu(p)}\Bigr)\Bigl(\prod_{p|e}\frac{p}{\nu(p)}\Bigr)\Bigl(\prod_{p|de}\frac{\nu(p)}{p}\Bigr).
\end{align}
The sum over $d,e$ vanishes unless $r=s$, and so we find that
\begin{align}
\sum_{d,e}\lambda_d\lambda_e\prod_{p|de}\frac{\nu(p)}{p}=\sum_{r}y_r^2\prod_{p|r}\frac{\nu(p)}{p-\nu(p)}.\label{eq:YQuad}
\end{align}
We now introduce new variables $y_r^*$ to perform an analogous diagonalization for the first quadratic form. Let $y_r^*$ be zero unless $r\in\langle P_3\rangle$, $(r,Wp_0)=1$, $r<X^{1/10}$ and $\nu(p)>1$ for all $p|r$. In this case, let $y_r^*$  be given by
\begin{equation}
y_r^*=\mu(r)\Bigl(\prod_{p|r}\frac{p-\nu(p)}{\nu(p)-1}\Bigr)\sum_{d:r|d}\lambda_d\Bigl(\prod_{p|d}\frac{\nu(p)-1}{p-1}\Bigr).\label{eq:YsDef}
\end{equation}
We see from this that for any $d$ satisfying the same conditions, we have
\begin{align}
\mu(d)\lambda_d&=\Bigl(\prod_{p|d}\frac{p-1}{\nu(p)-1}\Bigr)\sum_{e:d|e}\lambda_e\Bigl(\prod_{p|e}\frac{\nu(p)-1}{p-1}\Bigr)\sum_{r:d|r,\,r|e}\mu(r)\nonumber\\
&=\Bigl(\prod_{p|d}\frac{p-1}{\nu(p)-1}\Bigr)\sum_{r:d|r}y_r^*\Bigl(\prod_{p|r}\frac{\nu(p)-1}{p-\nu(p)}\Bigr).
\end{align}
This gives $\lambda_d$ in terms of $y_r^*$. Performing the analogous computation to \eqref{eq:YQuad} with $y_r^*$ in place of $y_r$, we obtain
\begin{equation}
\sum_{d,e}\lambda_d\lambda_e\prod_{p|de}\frac{\nu(p)-1}{p-1}=\sum_{r}(y^*_r)^2\prod_{p|r}\frac{\nu(p)-1}{p-\nu(p)}.\label{eq:YsQuad}
\end{equation}

Substituting the expression for $\lambda_d$ in terms of $y_r$ from the statement of the lemma in \eqref{eq:YsDef} gives (for $r$ such that $y_r^*\ne 0$)
\begin{align}
y_r^*&=\mu(r)\Bigl(\prod_{p|r}\frac{p-\nu(p)}{\nu(p)-1}\Bigr)\sum_{s:r|s}y_{s}\Bigl(\prod_{p|s}\frac{\nu(p)}{p-\nu(p)}\Bigr)\sum_{d:r|d,\,d|s}\mu(d)\Bigl(\prod_{p|d}\frac{p(\nu(p)-1)}{\nu(p)(p-1)}\Bigr)\nonumber\\
&=\mu(r)^2 r\sum_{s:r|s}\frac{y_s}{\phi(s)}.\label{eq:YsExpression}
\end{align}
We can now use Lemma \ref{lemma:Summation} and definition of $y_r$ from the statement of the lemma to evaluate this sum. We take $\gamma_1(p)=1$ if $p\equiv 3\pmod{4}$ and  $(p,rp_0W)=1$, and take $\gamma_1(p)=0$ otherwise. We see that for any $2\le w\le z\le X^{1/10}$ we have
\begin{equation}-\frac{1}{2}\sum_{p|rp_0W}\frac{\log{p}}{p}+O(1)\le \sum_{w<z<X^{1/10}}\frac{\gamma_1(p)\log{p}}{p}-\frac{1}{2}\log{z/w}\ll 1.\end{equation}
The sum on the left hand side is $O(\log\log{X})$. Thus, by applying Lemma \ref{lemma:Summation} to the sum \eqref{eq:YsExpression}, we have for squarefree $r<X^{1/10}$ with $r\in\langle P_3\rangle$ and $(r,Wp_0)=1$ and $\prod_{p|r}(\nu(p)-1)>0$ that
\begin{align}
&y_r^*=\frac{r}{\phi(r)}\sum_{t<X^{1/10}/r}\mu^2(t)\prod_{p|t}\frac{\gamma_1(p)}{p-\gamma_1(p)}\nonumber\\
&=\frac{r}{\phi(r)\Gamma(3/2)}\Biggl(\Bigl(\log{\frac{X^{\frac{1}{10}}}{r}}\Bigr)^{\frac{1}{2}}+O\Bigl(\frac{\log\log{X}}{(\log{X})^{\frac{1}{2}}}\Bigr)\Biggr)\prod_{p<X}\Bigl(1-\frac{\gamma_1(p)}{p}\Bigr)^{-1}\Bigl(1-\frac{1}{p}\Bigr)^{\frac{1}{2}}.\label{eq:YsVal}
\end{align}
Since $r,W\in\langle P_3\rangle$ and $\phi(p_0)/p_0=1+o(1)$, the product simplifies to give
\begin{equation}
\frac{r}{\phi(r)}\prod_{p<X}\Bigl(1-\frac{\gamma_1(p)}{p}\Bigr)^{-1}\Bigl(1-\frac{1}{p}\Bigr)^{1/2}=(1+o(1))\frac{\phi(W)}{W}\prod_{p<X}\Bigl(1-\frac{1}{p}\Bigr)^{\chi(p)/2},\label{eq:YsProd}\end{equation}
where $\chi$ is the non-trivial character modulo 4.

We can now evaluate both of the expressions \eqref{eq:YQuad} and \eqref{eq:YsQuad} using Lemma \ref{lemma:Summation}. For \eqref{eq:YQuad}, we take $\gamma_2(p)=\nu(p)$ if $p\equiv 3\pmod{4}$ and $(p,Wp_0)=1$, and $\gamma_2(p)=0$ otherwise. We see that
\begin{equation}
-\frac{k}{2}\sum_{p|p_0 W\prod_{i\ne j} (a_ib_j-b_ja_i)}\frac{\log{p}}{p}+O(1)\ge \sum_{w<p<z}\frac{\gamma_2(p)\log{p}}{p}-\frac{k}{2}\log\frac{z}{w}\ll 1.
\end{equation}
Thus, we obtain (with $\kappa=k/2$ and $L\ll k\log\log{X}$)
\begin{align}
&\sum_{r<X^{1/10}}y_r^2\prod_{p|r}\frac{\nu(p)}{p-\nu(p)}\nonumber\\
&=\frac{(\log{X^{1/10}})^{k/2}}{\Gamma(1+k/2)}\Bigl(1+O\Bigl(\frac{k^2\log\log{X}}{\log{X}}\Bigr)\Bigr)\prod_{p<X}\Bigl(1-\frac{\gamma_2(p)}{p}\Bigr)^{-1}\Bigl(1-\frac{1}{p}\Bigr)^{k/2}.\label{eq:YSum}
\end{align}
Similarly, for \eqref{eq:YsQuad} we take $\gamma_3(p)=(\nu(p)-1)p/(p-1)$ for $p\equiv 3\pmod{4}$ and $(p,Wp_0)=1$, and take $\gamma_3(p)=0$ otherwise. We obtain (with $\kappa=(k-1)/2$ and $L\ll k\log\log{X}$)
\begin{align}
&\sum_{r<X^{1/10}}\mu^2(r)\Bigl(\log{\frac{X^{1/10}}{r}}+O(\log\log{X})\Bigr)\prod_{p|r}\frac{\gamma_3(p)}{p-\gamma_3(p)}\nonumber\\
&=\frac{(\log{X^{1/10}})^{\frac{k+1}{2}}}{\Gamma((k+3)/2)}\Bigl(1+O\Bigl(\frac{k^2\log\log{X}}{\log{X}}\Bigr)\Bigr)\prod_{p<X}\Bigl(1-\frac{\gamma_3(p)}{p}\Bigr)^{-1}\Bigl(1-\frac{1}{p}\Bigr)^{\frac{k-1}{2}}.\label{eq:YsSum}
\end{align}
We note that if $\gamma_3(p)\ne 0$, then $1-\gamma_3(p)/p=(1-\gamma_2(p)/p)(1-1/p)^{-1}$. Thus, since $W\in\langle P_3\rangle$ and $\phi(p_0)/p_0=1+o(1)$, we have
\begin{align}
\prod_{p<X}\Bigl(1-\frac{\gamma_3(p)}{p}\Bigr)^{-1}\Bigl(1-\frac{1}{p}\Bigr)^{\frac{k-1}{2}}\sim\frac{W}{\phi(W)}\prod_{p<X}\Bigl(1-\frac{\gamma_2(p)}{p}\Bigr)^{-1}\Bigl(1-\frac{1}{p}\Bigr)^{\frac{k-\chi(p)}{2}}.\label{eq:YProd}
\end{align}
Putting the estimates \eqref{eq:YsVal}, \eqref{eq:YsProd}, \eqref{eq:YSum}, \eqref{eq:YsSum}, \eqref{eq:YProd} together, we obtain
\begin{align}
\sum_{r}(y^*_r)^2\prod_{p|r}\frac{\nu(p)-1}{p-\nu(p)}&\sim\frac{\phi(W)\Gamma(1+k/2)(\log{X^{1/10}})^{1/2}}{W\Gamma(3/2)^2\Gamma((k+3)/2)}\prod_{p<X}\Bigl(1-\frac{1}{p}\Bigr)^{\chi(p)/2}\nonumber\\
&\qquad\times\sum_{r}y_r^2\prod_{p|r}\frac{\nu(p)}{p-\nu(p)}.
\end{align}
Stirling's formula shows that $\Gamma(x)\sim \sqrt{2\pi x}x^xe^{-x}$, and so for $k$ sufficiently large $\Gamma(1+k/2)/\Gamma((k+3)/2)\gg k^{-1/2}$. Since the product over primes is convergent as $X\rightarrow\infty$, this gives
\begin{align}
\sum_{r}(y^*_r)^2\prod_{p|r}\frac{\nu(p)-1}{p-\nu(p)}
&\gg \frac{\phi(W)}{W}k^{-1/2}(\log{X})^{1/2}\sum_{r}y_r^2\prod_{p|r}\frac{\nu(p)}{p-\nu(p)}.
\end{align}
Finally, from \eqref{eq:YSum} and $k\le (\log{X})^{1/5}$, it follows that we have the crude bound that this is $\gg 1$.
\end{proof}
\begin{lemma}\label{lemma:LambdaSize}
Let $k\le (\log{X})^{1/5}$ and $\lambda_d$ be as given by Lemma \ref{lemma:Ratio}. Then
\[|\lambda_d|\ll (\log{X})^{k/2}.\]
\end{lemma}
\begin{proof}
This is an immediate application of Lemma \ref{lemma:Summation}. We take $\gamma_4(p)=\nu(p)$ if $p\equiv3\pmod{4}$ and $(p,dWp_0)=1$, and $\gamma_4(p)=0$ otherwise. By Lemma \ref{lemma:Summation} (taking $\kappa=k/2$ and $L\ll k\log\log{X}$) we have
\begin{align}
|\lambda_d|&=\Bigl(\prod_{p|d}\frac{p}{p-\nu(p)}\Bigr)\sum_{\substack{r<X^{1/10}/d\\ (r,dWp_0)=1\\ r\in\langle P_3\rangle}}\mu^2(r)\prod_{p|r}\frac{\nu(p)}{p-\nu(p)}\nonumber\\
&\ll \frac{(\log{X})^{k/2}}{\Gamma(1+k/2)}
\prod_{p|d}\Bigl(1-\frac{\nu(p)}{p}\Bigr)^{-1}\prod_{p}\Bigl(1-\frac{\gamma_4(p)}{p}\Bigr)^{-1}\Bigl(1-\frac{1}{p}\Bigr)^{k/2}\nonumber\\
&\ll \frac{(\log{X})^{k/2}}{\Gamma(1+k/2)}\prod_{\substack{(\log{X})^{1/3}<p<X\\ p\equiv 3\pmod{4}}}\Bigl(1-\frac{k}{p}\Bigr)^{-1}\Bigl(1-\frac{1}{p}\Bigr)^{k/2}
\prod_{\substack{(\log{X})^{1/3}<p<X\\ p\equiv 1\pmod{4}}}\Bigl(1-\frac{1}{p}\Bigr)^{k/2}.
\end{align}
Recalling that $k\le (\log{X})^{1/5}$, we see that this is
\begin{align}
&\ll \frac{(\log{X})^{k/2}}{\Gamma(1+k/2)}\exp\Bigl(\frac{k}{2}\sum_{\substack{(\log{X})^{1/3}<p<X\\ p\equiv 3\pmod{4}}}\frac{1+O(p^{-1})}{p}-\frac{k}{2}\sum_{\substack{(\log{X})^{1/3}<p<X\\ p\equiv 1\pmod{4}}}\frac{1+O(p^{-1})}{p}\Bigr)\nonumber\\
&\ll \exp(O(k))k^{-k/2}(\log{X})^{k/2}\ll (\log{X})^{k/2}.\end{align}
\end{proof}
\begin{proof}[Proof of Theorem \ref{theorem:Admiss}]
By Lemmas \ref{lemma:Ratio} and \ref{lemma:LambdaSize}, there is a choice of coefficients $\lambda_d$ such that $|\lambda_d|\ll (\log{X})^{k/2}$ and 
\[\sum_{d,e}\lambda_d\lambda_e\prod_{p|de}\frac{\nu(p)-1}{p-1}\gg k^{-1/2}(\log{X})^{1/2}\frac{\phi(W)}{W}\sum_{d,e}\lambda_d\lambda_e\prod_{p|de}\frac{\nu(p)}{p}\gg 1.\]
Thus, by Lemmas \ref{lemma:Const} and \ref{lemma:Squares}, this choice of $\lambda_d$ corresponds to a choice of $w_n\ge 0$ such that
\begin{equation}
\sum_{\substack{X< n\le 2X\\n\equiv v_0\pmod{W}}}\#\Bigl\{1\le i\le k:L_i(n)\in\mathcal{S}\Bigr\}w_n\gg k^{1/2}\sum_{\substack{X< n\le 2X\\n\equiv v_0\pmod{W}}}w_n\gg \frac{X}{W\log{X}}.\label{eq:KeySum}
\end{equation}
Lemma \ref{lemma:LambdaSize} shows $|\lambda_d|\ll (\log{X})^{k/2}$, so we have
\begin{equation}
w_n\ll (\log{X})^k\Bigl(\sum_{\substack{d|\prod L_i(n)\\ \lambda_d\ne 0}}1\Bigr)^2.
\end{equation}
Thus
\begin{align}
\sum_{X\le n\le 2X}w_n^2&\ll (\log{X})^{2k}\sum_{d_1,d_2,d_3,d_4<X^{1/10}}\Bigl(\prod_{i=1}^4\mu^2(d_i)\Bigr) \sum_{\substack{X\le n\le 2X\\ d_1,d_2,d_3,d_4|\prod_{i=1}^kL_i(n)}}1\nonumber\\
& \ll X (\log{X})^{2k}\sum_{d_1,d_2,d_3,d_4<X^{1/10}}\Bigl(\prod_{i=1}^4\mu^2(d_i)\Bigr)\prod_{p|d_1d_2d_3d_4}\frac{k}{p} \nonumber\\
&\ll X (\log{X})^{2k}\sum_{r<X^{2/5}} \frac{(15 k^2)^{\nu(r)}}{r}
\ll X(\log{x})^{16 k^2}.\label{eq:NormalisedSum}
\end{align}
By Cauchy-Schwarz, we have that
\begin{align}
&\#\{n\in[X,2X]:\#\{L_1(n),\dots,L_k(n)\}\cap\mathcal{S}\ge ck^{1/2}\}\nonumber\\
&\quad\ge \Bigl(k^{-1}\sum_{n\in[X,2X]}\Bigl(\#\{i:L_i(n)\in\mathcal{S}\}-ck^{1/2}\Bigr)w_n\Bigr)^2\Bigl(\sum_{n\in[X,2X]}w_n^2\Bigr)^{-1},
\end{align}
provided the sum being squared is positive. Therefore, by \eqref{eq:KeySum} and \eqref{eq:NormalisedSum}, we see that for $c$ sufficiently small we have
\begin{align}
&\#\{n\in[X,2X]:\#\{L_1(n),\dots,L_k(n)\}\cap\mathcal{S}\ge ck^{1/2}\}\nonumber\\
&\quad\ge \frac{1}{X(\log{X})^{16k^2}}\Biggl(\sum_{\substack{n\in[X,2X]\\ n\equiv v_0\pmod{W}}}\Bigl(\#\Bigl\{i:L_i(n)\in\mathcal{S}\Bigr\}w_n-ck^{1/2}\Bigr)w_n\Biggr)^2\nonumber\\
&\quad\gg \frac{X}{W^2(\log{X})^{16k^2+2}}\gg X\exp(-(\log{X})^{1/2}).
\end{align}
Here we have used the fact that $k\le (\log{X})^{1/5}$ and $W\ll \exp((\log{X})^{2/5})$ in the last line. This completes the proof of Theorem \ref{theorem:Admiss}.
\end{proof}
%
%
%
%
%
\section{Acknowledgement}
The work in this paper was started whilst the author was a CRM-ISM Postdoctoral Fellow and the Universit\'e de Montr\'eal, and was completed whilst he was a Fellow by Examination at Magdalen College, Oxford.


\begin{thebibliography}{10}

\bibitem{BalogWooley}
Antal Balog and Trevor~D. Wooley.
\newblock Sums of two squares in short intervals.
\newblock {\em Canad. J. Math.}, 52(4):673--694, 2000.

\bibitem{Davenport}
H.~Davenport.
\newblock {\em Multiplicative Number Theory}.
\newblock Graduate Texts in Mathematics. Springer New York, 2000.

\bibitem{FGKT}
K.~Ford, B.~Green, S.~Konyagin, and T.~Tao.
\newblock Large gaps between consecutive primes.
\newblock preprint, \url{http://arxiv.org/abs/1408.4505}.

\bibitem{FordPratt}
Kevin Ford, Sergei~V. Konyagin, and Florian Luca.
\newblock Prime chains and {P}ratt trees.
\newblock {\em Geom. Funct. Anal.}, 20(5):1231--1258, 2010.

\bibitem{FriedlanderGranvilleIII}
John Friedlander and Andrew Granville.
\newblock Limitations to the equi-distribution of primes. {III}.
\newblock {\em Compositio Math.}, 81(1):19--32, 1992.

\bibitem{GGPY}
D.~A. Goldston, S.~W. Graham, J.~Pintz, and C.~Y. Y{\i}ld{\i}r{\i}m.
\newblock Small gaps between products of two primes.
\newblock {\em Proc. Lond. Math. Soc. (3)}, 98(3):741--774, 2009.

\bibitem{GPYII}
Daniel~A. Goldston, J{\'a}nos Pintz, and Cem~Yal{\c{c}}in Y{\i}ld{\i}r{\i}m.
\newblock Primes in tuples. {II}.
\newblock {\em Acta Math.}, 204(1):1--47, 2010.

\bibitem{GranvilleSound}
Andrew Granville and K.~Soundararajan.
\newblock An uncertainty principle for arithmetic sequences.
\newblock {\em Ann. of Math. (2)}, 165(2):593--635, 2007.

\bibitem{HalberstamRichert}
H.~Halberstam and H.E. Richert.
\newblock {\em Sieve methods}.
\newblock L.M.S. monographs. Academic Press, 1974.

\bibitem{HeathBrownIntervals}
D.~R. Heath-Brown.
\newblock The number of primes in a short interval.
\newblock {\em J. Reine Angew. Math.}, 389:22--63, 1988.

\bibitem{HooleyIV}
Christopher Hooley.
\newblock On the intervals between numbers that are sums of two squares. {IV}.
\newblock {\em J. Reine Angew. Math.}, 452:79--109, 1994.

\bibitem{Huxley}
M.~N. Huxley.
\newblock On the difference between consecutive primes.
\newblock {\em Invent. Math.}, 15:164--170, 1972.

\bibitem{MaierPomerance}
H.~Maier and C.~Pomerance.
\newblock Unusually large gaps between consecutive primes.
\newblock {\em Trans. Amer. Math. Soc.}, 322(1):201--237, 1990.

\bibitem{MaierStewart}
H.~Maier and C.~L. Stewart.
\newblock On intervals with few prime numbers.
\newblock {\em J. Reine Angew. Math.}, 608:183--199, 2007.

\bibitem{Maier}
Helmut Maier.
\newblock Primes in short intervals.
\newblock {\em Michigan Math. J.}, 32(2):221--225, 1985.

\bibitem{MaynardII}
J.~Maynard.
\newblock Dense clusters of primes in subsets.
\newblock preprint, \url{http://arxiv.org/abs/1405.2593}.

\bibitem{MaynardLargeGaps}
J.~Maynard.
\newblock Large gaps between primes.
\newblock preprint, \url{http://arxiv.org/abs/1408.5110}.

\bibitem{Maynard}
J.~Maynard.
\newblock Small gaps between primes.
\newblock {\em Ann. of Math.(2), to appear}.

\bibitem{Pintz}
J.~Pintz.
\newblock Very large gaps between consecutive primes.
\newblock {\em J. Number Theory}, 63(2):286--301, 1997.

\bibitem{Richards}
Ian Richards.
\newblock On the gaps between numbers which are sums of two squares.
\newblock {\em Adv. in Math.}, 46(1):1--2, 1982.

\bibitem{Zhang}
Y.~Zhang.
\newblock Bounded gaps between primes.
\newblock {\em Ann. of Math. (2)}, 179(3):1121--1174, 2014.

\end{thebibliography}
\end{document}